\documentclass[a4paper,reqno,11pt,oneside]{amsart}
\usepackage[pdftex]{graphicx,xcolor,hyperref}
\usepackage{amssymb}
\usepackage{amstext}
\usepackage{amsmath}
\usepackage{amscd}
\usepackage{amsthm}
\usepackage{amsfonts}
\usepackage{enumerate}
\usepackage{caption}
\usepackage{here}
\usepackage{bm}
\usepackage{tikz}
\usepackage{calc}
\usepackage{multirow}
\usepackage{makecell}
\usepackage[colorinlistoftodos]{todonotes}
\usetikzlibrary{decorations.markings}
\usetikzlibrary{positioning}
\tikzstyle{vertex}=[circle ,draw, inner sep=0pt, minimum size=6pt]

\usepackage{comment}
\usepackage[whole]{bxcjkjatype}

\makeatletter
  
  \@addtoreset{equation}{section}
\makeatother

\newcommand{\Bc}{\mathcal{B}}

\newcommand{\ZZ}{\mathbb{Z}}

\newcommand{\RR}{\mathbb{R}}

\newcommand{\bb}{\mathbf{b}}
\newcommand{\eb}{\mathbf{e}}

\newcommand{\qb}{\mathbf{q}}

\newcommand{\xb}{\mathbf{x}}
\newcommand{\yb}{\mathbf{y}}

\newcommand{\bigfrac}[2]{
\left\lfloor\frac{#1}{#2}\right\rfloor
}

\def\opn#1#2{\def#1{\operatorname{#2}}} 
\opn\conv{conv} \opn\dep{depth} \opn\Spec{Spec} \opn\cone{cone} \opn\ini{in} \opn\codeg{codeg} \opn\deg{deg}
\opn\Graph{Graph} \opn\sign{sign} \opn\Ehr{Ehr} \opn\rank{rank} \opn\type{type} \opn\reg{reg} \opn\core{core}
\opn\Hilb{Hilb} \opn\Indeg{Indeg} \opn\link{link} \opn\Tor{Tor} \opn\MNF{MNF} \opn\Stab{Stab}

\def\mindex{\mathrm{m\mbox{-}index}}

\newtheorem{thm}{Theorem}[section]
\newtheorem{cor}[thm]{Corollary}
\newtheorem{lem}[thm]{Lemma}
\newtheorem{prop}[thm]{Proposition}

\newtheorem{q}[thm]{Question}

\theoremstyle{definition}

\theoremstyle{remark}
\newtheorem{rem}[thm]{Remark}

\definecolor{suprise}{rgb}{0.0, 1, 1}

\begin{document}

\title{On the magic positivity of Ehrhart Polynomials of dilated polytopes}
\author{Masato Konoike}

\address[M. Konoike]{Department of Pure and Applied Mathematics, Graduate School of Information Science and Technology, Osaka University, Suita, Osaka 565-0871, Japan}
\email{kounoike-m@ist.osaka-u.ac.jp}


\maketitle
\begin{abstract} 
A polynomial $f(x)$ of degree $d$ is said to be magic positive if all the coefficients are non-negative when $f(x)$ is expanded with respect to the basis $\{x^i(x+1)^{d-i}\}_{i=0}^d$. It is known that if $f(x)$ is magic positive, then the polynomial appearing in the numerator of its generating function is real-rooted. 
In this paper, we show that for a polynomial $f(x)$ with positive real coefficients, there exists a positive real number $k$ such that $f(k'x)$ is magic positive for any $k' \geq k$.
Furthermore, for any integer $d\geq3$, we show the existence of a $d$-dimensional polytope $P$ such that the Ehrhart polynomial of $kP$ is not magic positive for a given integer $k$.
Finally, we investigate how much certain polytopes need to be dilated to make their Ehrhart polynomials magic positive.
\end{abstract}

\section{Introduction}
Let $f(x) = \sum_{i=0}^d a_i x^i$ be a polynomial of degree $d$ with $a_i \in \RR$. We recall the following properties:
\begin{itemize}
    \item We say that $f(x)$ is \textit{unimodal}, if there exist a positive integer $k$ such that $a_0\leq a_1\leq \cdots \leq a_k \geq \cdots \geq a_d$.
    \item We say that $f(x)$ is \textit{log-concave}, if $a_i^2\geq a_{i-1}a_{i+1}$ for all $i$.
    \item We say that $f(x)$ is \textit{real-rooted}, if all its roots are real.
\end{itemize}
It is well-known that real-rootedness implies log-concavity, and that if all the $a_i$ are positive, then log-concavity implies unimodality. 
These properties frequently appear in mathematics and are actively studied using a wide range of techniques from algebra, analysis, probability theory, and combinatorics. See e.g., \cite{branden2015unimodality,higashitani2019arithmetic,stanley1989log}.
A particularly important class of polynomials where these properties have been extensively investigated is the numerator polynomial of the generating function of $f(x)$. Specifically, if
\[
\sum_{n\geq 0}f(n)t^n = \frac{h(t)}{(1-t)^{d+1}},
\]
where $h(t) = h_0 + h_1t + \cdots + h_dt^d$ is a nonzero polynomial of degree at most $d$ with real coefficients, then $h(t)$ can be viewed as a transformation of $f(x)$ under a change of basis. Indeed, since
\[
f(x) = \sum_{i = 0}^{d}h_i\binom{x+d-i}{d},
\]
the polynomial $h(t)$ encodes $f(x)$ in a binomial coefficient basis.
As one approach to study the real-rootedness of $h(t)$, we expand the polynomial $f(x)$ in a different basis as follows:
\[
f(x) = \sum_{i=0}^d b_i x^i (x+1)^{d-i}.
\]
If $b_0,\ldots,b_d\geq 0$, then $f(x)$ is said to be \textit{magic positive}. The term `magic positive' was introduced by Ferroni and Higashitani \cite{ferroni2024examples}. The following theorem shows that the magic positivity of a polynomial is useful for studying the real-rootedness of $h(t)$.

\begin{thm}[{\cite{beck2019h,branden2004linear,ferroni2024examples}}]
    Let $c_0, \ldots, c_d \geq 0$ and let $h(t) \in \RR[t]$ be the polynomial such that
    \[
    \sum_{n\geq0}\left(\sum_{j=0}^{d}c_jn^j(n+1)^{d-j}\right)t^n = \frac{h(t)}{(1-t)^{d+1}}.
    \]
    Then $h(t)$ is real-rooted.
\end{thm}
Our main motivation comes from Ehrhart theory.
A \textit{lattice polytope} is the convex hull of finitely many elements in a lattice contained in $\mathbb{R}^d$, typically $\mathbb{Z}^d$. A lattice polytope $P$ is called \textit{reflexive} if
\[
P^\ast := \{ y \in \RR^d : \langle x, y\rangle \leq 1 \text{ for any } x \in P\}
\]
is also a lattice polytope, where $\langle \cdot, \cdot \rangle$ denotes the usual inner product of $\RR^d$. 

By a theorem of Ehrhart \cite{ehrhart1962sur}, $|nP\cap \mathbb{Z}^d|$ is given by a polynomial $E_P(n)$ of degree $\dim P$ in $n$ for all integers $n > 0$. The polynomial $E_P(n)$ is called the \textit{Ehrhart polynomial} of $P$.
If all coefficients of $E_P(n)$ are positive, then we call $P$ \textit{Ehrhart positive}. Ehrhart positivity of the polytope does not hold in general. There has been active research on characterizing which polytopes are Ehrhart positive. See, e.g., \cite{ferroni2021hypersimplices,liu2019positivity,postnikov2009permutohedra}.
It is known that the generating function $1 + \sum_{n\geq1}E_P(n)t^n$ of the Ehrhart polynomial takes the form of a rational function as follows:
\[
1+\sum_{n\geq1}E_P(n)t^n = \frac{h_0^\ast + h_1^\ast t + \cdots + h_d^\ast t^d}{(1-t)^{d+1}}.
\]
The numerator polynomial of this generating function is called the \textit{$h^\ast$-polynomial} of $P$, denoted by $h^\ast_P(t)$.
A fundamental theorem by Stanley \cite{stanley1980decompositions} states that the coefficients of the $h^\ast$-polynomial are always non-negative integers. It was proved by Hibi~\cite{hibi1992dual} that a $d$-dimensional lattice polytope $P$ is reflexive if and only if its $h^\ast$-polynomial is palindromic and has degree $d$, that is, $h^\ast _P (t)=t^dh^\ast _P \left(\frac{1}{t}\right)$.

In this paper, we consider the following three fundamental questions.
\begin{q}\label{question}
    Let $P$ be a $d$-dimensional lattice polytope.
    \begin{itemize}
        \item [(i)] Does there exist a positive integer $k$ such that $E_{kP}(n)$ is magic positive?
        \item [(ii)] If $E_{kP}(n)$ is magic positive, is $E_{(k+1)P}(n)$ also magic positive?
        \item [(iii)] In each dimension $d$, does there exist a positive integer $k$ such that $E_{kP}(n)$ is magic positive for any $d$-dimensional lattice polytope $P$?
    \end{itemize}
\end{q}
As a similar question, whether a dilated polytope has the same properties of the original polytope (such as IDP and unimodular triangulation) has been actively studied. See, e.g., \cite{bruns1997normal, David2014integer, jochemko2022symmetric, liu2021unimodular}.

The following first main result provides an answer to Question~\ref{question} (i). That is, it shows that if a polytope which is Ehrhart positive is sufficiently dilated, then the Ehrhart polynomial of the dilated polytope becomes magic positive.
\begin{thm}\label{main1}
    Let $f(x)$ be a polynomial with positive real coefficients. Then, there exists a positive real number $k$ such that $f(kx)$ is magic positive.
\end{thm}

The following second main result provides an answer to Question~\ref{question} (ii). That is, it shows that once the Ehrhart polynomial of a dilated polytope becomes magic positive, it remains magic positive under all subsequent dilations.
\begin{thm}\label{main2}
    Let $f(x)$ be a polynomial with positive real coefficients. If there exists a positive real number $k$ such that $f(kx)$ is magic positive, then so is $f(k'x)$ for any $k' \geq k$.
\end{thm}
The following third main result provides an answer to Question~\ref{question} (iii). That is, it shows that in each dimension $d \geq 3$, there does not exist a fixed integer $k>0$ such that the Ehrhart polynomial of $kP$ is magic positive.
\begin{thm}\label{main3}
    Let $P$ be a lattice polytope of dimension $d$ which is Ehrhart positive. Then we have the following: 
    \begin{itemize}
        \item [(1)] For $d = 2$, $E_{2P}(n)$ is always magic positive.
        \item [(2)] For any $d \geq 3$ and $k \in \ZZ_{>0}$, there exists a lattice polytope $P$ such that $E_{kP}(n)$ is not magic positive.
    \end{itemize}
\end{thm}

From Theorem 1.3, if a polytope which is Ehrhart positive is sufficiently dilated, then its Ehrhart polynomial has magic positivity. Moreover, from Theorem 1.4, once the Ehrhart polynomial has magic positivity due to the dilation of the polytope, it remains magic positive thereafter. From these, it is natural to define the following invariant:
\[
\mindex(P) := \min\{k \in \ZZ_{> 0}: \text{$E_{kP}(n)$ is magic positive}\}.
\]
In Section~\ref{example}, we explicitly compute $\mindex(P)$ for several polytopes.

\subsection*{Acknowledgements}
The author would like to thank Akihiro Higashitani for his helpful comments and advice on improving this paper. 
Moreover, the author would like to thank Luis Ferroni for suggesting the notation `$\mindex$' and for proposing the content of Section~\ref{CL}.

\section{Proof of Theorem~\ref{main1}}\label{sec:main1}
In this section, we give a proof of Theorem~\ref{main1}. 
For the proof of Theorem~\ref{main1}, we use the following lemma.
\begin{lem}\label{important}
    Let $f(x) = \sum_{i=0}^{d}b_i x^i$ be a polynomial of degree $d$. Then we have $f(kx) = \sum_{i=0}^{d}(\sum_{j=0}^{i} (-1)^{i-j}\binom{d-j}{i-j}b_jk^j) x^i(x+1)^{d-i}$.
\end{lem}
\begin{proof}
    Let $a_i$ denote the coefficient of $x^i(x+1)^{d-i}$ when $f(kx)$ is expanded with respect to the basis $\{x^i(x+1)^{d-i}\}_{i=0}^d$. Then, the transformation from the sequence $(b_0, kb_1, \ldots, k^db_d)$ to the sequence $(a_0, \ldots, a_d)$ is the same type of transformation as that from the $f$-vector to the $h$-vector of a simplicial complex (see \cite[Exercise 11.2]{hibi1992algebraic}). Therefore, we obtain $a_i = \sum_{j=0}^{i} (-1)^{i-j}\binom{d-j}{i-j}b_jk^j$.
\end{proof}

By using Lemma~\ref{important}, we can prove Theorem~\ref{main1}.
\begin{proof}[Proof of Theorem~\ref{main1}]
    From Lemma~\ref{important}, the coefficient of $x^i(x+1)^{d-i}$ when $f(kx)$ is expanded with respect to the basis $\{x^i(x+1)^{d-i}\}_{i=0}^d$ is given by $g_i(k) := \sum_{j=0}^{i} (-1)^{i-j}\binom{d-j}{i-j}b_jk^j$. By assumption, $f(x)$ has positive real coefficients, that is, $b_i > 0$ for all $i$. This implies that $g_i(k) \geq 0$ for all $i$ when $k$ is sufficiently large. Therefore, there exists a positive real number $k$ such that $f(kx)$ is magic positive.
\end{proof}

\section{Proof of Theorem~\ref{main2}}\label{sec:main2}
In this section, we give a proof of Theorem~\ref{main2}.
For the proof of Theorem~\ref{main2}, we use the following variant of Farkas' lemma.
\begin{lem}[{\cite[Proposition 6.4.3]{Matouek2006under}}]\label{Farkas'}
    Let $A \in \RR^{m \times n}$ and $\bb \in \RR^m$. Then exactly one of the following two assertions is true:
    \begin{itemize}
        \item[(i)] There exists an $\xb \in \RR^n$ such that $A\xb \leq \bb$ and $\xb \geq 0$, or
        \item[(ii)] There exists a $\yb \in \RR^m$ such that $\yb^T A \geq 0, \yb \geq 0$ and $\yb^T \bb < 0$.
    \end{itemize}
\end{lem}
By using Lemma~\ref{Farkas'}, we can prove Theorem~\ref{main2}.
\begin{proof}[Proof of Theorem~\ref{main2}]
Suppose that $f(kx)$ is magic positive and $f(k'x)$ is not magic positive for some $k' > k$.
That is, by Lemma~\ref{important}, $\sum_{j=0}^{i} (-1)^{i-j}\binom{d-j}{i-j}b_jk^j \geq 0$ holds for any $i = 0, \ldots, d$, and there exists some $\ell \in \{1, \ldots, d\}$ such that $\sum_{j=0}^{\ell} (-1)^{\ell-j}\binom{d-j}{\ell-j}b_j(k')^j < 0$ holds. Under the notation of Lemma~\ref{Farkas'}, we define $A \in \RR^{(\ell+1) \times \ell}$ and $\bb\in\RR^{\ell+1}$ as follows:
\[
\resizebox{\textwidth}{!}{$
\begin{array}{rl}
A = &
\begin{pmatrix}
    -k & 0 & \cdots & 0 & 0\\
    \binom{d-1}{1}k & -k^2 & \cdots & 0 & 0 \\
    \vdots & \vdots & \ddots & \vdots & \vdots \\
    (-1)^\ell \binom{d-1}{\ell-1}k & (-1)^{\ell-1} \binom{d-2}{\ell-2}k^2 & \cdots & \binom{d-\ell+1}{1}k^{\ell-1} & -k^\ell \\
    (-1)^{\ell-1} \binom{d-1}{\ell-1}k' & (-1)^{\ell-2} \binom{d-2}{\ell-2}(k')^2 & \cdots & -\binom{d-\ell+1}{1}(k')^{\ell-1} & (k')^\ell
\end{pmatrix}
\hspace{5pt} \text{and} \hspace{5pt}
\mathbf{b} = 
\begin{pmatrix}
    -d\\
    \binom{d}{2} \\
    \vdots \\
    (-1)^\ell \binom{d}{\ell} \\
    (-1)^{\ell+1} \binom{d}{\ell}
\end{pmatrix},
\end{array}
$}
\]
where the $(i, j)$-entry of $A$ is $(-1)^{i-j+1}\binom{d-j}{i-j}k^j$ for $1 \leq i \leq \ell$, the $(\ell+1, j)$-entry of $A$ is $(-1)^{\ell-j}\binom{d-j}{\ell-j}(k')^j$, the $i$-th entry of the vector $\bb$ is $(-1)^i\binom{d}{i}$ for $1 \leq i \leq \ell$ and the $(\ell+1)$-th entry of $\bb$ is $(-1)^{\ell+1}\binom{d}{\ell}$. Let $\xb = (b_1, \ldots, b_\ell)^T$. Then we see that the condition (i) of Lemma~\ref{Farkas'} holds. For $1 \leq i \leq \ell$, let $y_i = \binom{d-i}{\ell-i}(k'-k)^{\ell-i}(k')^i$ and $y_{\ell+1} = k^\ell$, where $y_j$ denotes the $j$-th entry of the vector $\yb$ for $1 \leq j \leq \ell + 1$. Then, the $j$-th entry of the vector $\yb^T A$ is the following: 

\begin{align*}
    \yb^T A &= \sum_{i=1}^{\ell} \left(\binom{d-i}{\ell-i}(k'-k)^{\ell-i}(k')^i \cdot (-1)^{i-j+1}\binom{d-j}{i-j}k^j\right) + k^\ell \cdot (-1)^{\ell-j}\binom{d-j}{\ell-j}(k')^j \\
    &= \sum_{i=0}^{\ell-j} \left((-1)^{i+1}\binom{d-j}{i}k^j \binom{d-i-j}{d-\ell}(k'-k)^{\ell-i-j}(k')^{i+j}\right) + (-1)^{\ell-j}\binom{d-j}{\ell-j}(k')^jk^\ell \\
    &= k^j(k')^j\binom{d-j}{\ell-j}\left(\sum_{i=0}^{\ell-j} \left((-1)^{i+1}\binom{\ell-j}{i}(k'-k)^{\ell-i-j}(k')^{i}\right) + (-1)^{\ell-j}k^{\ell-j}\right) \\
    &= k^j(k')^j\binom{d-j}{\ell-j}(-((k'-k)-k')^{\ell - j} + (-1)^{\ell-j}k^{\ell-j}) = 0.
\end{align*}
Furthermore, $\yb^T \bb$ is given as follows:
\begin{align*}
    \yb^T \bb &= \sum_{i=1}^{\ell} \binom{d-i}{\ell-i}(k'-k)^{\ell-i}(k')^i \cdot (-1)^i\binom{d}{i} + (-1)^{\ell+1}\binom{d}{\ell}k^\ell \\
    &= -\binom{d}{\ell}\left(\sum_{i=1}^{\ell}(-1)^{i+1}\binom{\ell}{i}(k'-k)^{\ell-i}(k')^i + (-1)^{\ell} k^\ell\right) \\
    &=-\binom{d}{\ell}\left((k'-k)^\ell+(-1)^{\ell+1}k^\ell+(-1)^\ell k^\ell\right)\\
    &= -\binom{d}{\ell}(k'-k)^{\ell} < 0.
\end{align*}
From these calculations and the fact that $\yb \geq 0$, the condition (ii) of Lemma~\ref{Farkas'} follows. This contradicts the fact that (i) and (ii) of Lemma~\ref{Farkas'} cannot hold simultaneously. Therefore, $f(k'x)$ is magic positive for any $k' \geq k$.
\end{proof}

\section{Proof of Theorem~\ref{main3}}\label{sec:main3}
In this section, we give a proof of Theorem~\ref{main3}. 

\begin{proof}[Proof of Theorem~\ref{main3}]
\begin{itemize}
    \item [(1)] 
    Let $P$ be a lattice polytope of dimension 2. Then we have 
    \[
    E_{kP}(n) = \binom{kn+2}{2} + a\binom{kn+1}{2} + b\binom{kn}{2}
    \]
    where $a$ and $b$ are non-negative integers satisfying $a \geq b$ (see \cite[Proposition 36.2]{hibi1992algebraic}). This expression can be rewritten as follows:
    \[
    E_{kP}(n) = (n+1)^2 + \frac{1}{2}((a-b+3)k-4)n(n+1)+\frac{1}{2}((a+b+1)k^2-(a-b+3)k+2)n^2.
    \]
    Since $a,b \geq0$, we see that 
    \[
    \frac{1}{2}\left((a+b+1)k^2-(a-b+3)k+2\right) = \binom{k-1}{2} + a\binom{k}{2} + b\binom{k+1}{2} \geq 0
    \]
    for all $k\geq1$. Therefore, $E_{kP}(n)$ is magic positive if and only if $(a-b+3)k-4 \geq 0$. That is, it suffices to have $k \geq \frac{4}{a-b+3}$. Since $a - b \geq 0$, $\frac{4}{a-b+3}$ takes its maximum value $\frac{4}{3}$ when $a = b$. Hence, we conclude that $E_{2P}(n)$ is magic positive.
    \item[(2)] 
    Let
    \[
    P_{q,d} = \conv(\eb_1, \ldots, \eb_d, -q(\eb_1 + \cdots + \eb_d))
    \]
    where $\eb_i$ denotes the $i$-th unit vector of $\RR^d$. The $h^\ast$-polynomial of this polytope is computed in \cite{higashitani2010shifted} and is given by $h^\ast_{P_{q,d}}(t) = 1 + qt + \cdots + qt^d$. Thus, the Ehrhart polynomial is
    \[
    E_{P_{q,d}}(n) = \binom{n+d}{d} + \sum_{i=1}^{d} q\binom{n+d-i}{d}.
    \]
    First, we confirm that $P_{q,d}$ is Ehrhart positive. It is clear that all the coefficients of $\binom{n+d}{d}$ are positive.
    Moreover, there exist $a_0, \ldots, a_{\bigfrac{d}{2}} > 0$ such that the following equation holds: 
    \begin{align}
        \sum_{i=1}^{d}\binom{n+d-i}{d} = \binom{n+d}{d+1} - \binom{n}{d+1} = \sum_{j=0}^{\bigfrac{d-1}{2}} a_j n^{d-2j}.
    \end{align}    
    Therefore, $P_{q,d}$ is Ehrhart positive. 

    Let $E_{P_{q,d}}(kn) = \sum_{i=0}^{d}c_in^i(n+1)^{d-i}$.
    When $d$ is an odd integer with $d \geq 3$, the coefficient $c_2$ takes the following form by Lemma~\ref{important} and the above discussion:
    \[
    c_2 = b_2k^2-\binom{d-1}{1}b_1k+\binom{d}{2}
    \]
    where $b_2$ is constant and $b_1$ is a linear polynomial in $q$ with positive leading coefficient from (4.1). Hence, as $q$ becomes large, the value of $k$ required to make $c_2$ positive also increases. Therefore, Theorem~\ref{main3} holds when $d$ is odd.

    When $d$ is an even integer with $d \geq 4$, the coefficient $c_3$ takes the following form by Lemma~\ref{important} and the above discussion:
    \[
    c_3 = b_3k^3 - \binom{d-2}{1}b_2k^2 + \binom{d-1}{2}b_1k - \binom{d}{3}
    \]
    where $b_3, b_1$ are constant, and $b_2$ is a linear polynomial in $q$ with positive leading coefficient from (4.1). By a similar argument to the odd case, Theorem~\ref{main3} also holds when $d$ is even.
\end{itemize}
\end{proof}

\section{The Case of Some Dilated Polytopes}\label{example}
From Theorem~\ref{main1} and Theorem~\ref{main2}, it is natural to define the following invariant for a lattice polytope $P$ which is Ehrhart positive:
\[
\mindex(P) := \min\{k \in \ZZ_{>0}: \text{$E_{kP}(n)$ is magic positive}\}.
\]
If there exists a positive integer $k$ such that $E_{kP}(n)$ is magic positive, then it follows that $E_P(n)$ is Ehrhart positive. Therefore, we see that computing $\mindex(P)$ is more difficult than proving the Ehrhart positivity of $P$.

In general, if the volume of a lattice polytope $P$ is less than 1, since the constant term of the Ehrhart polynomial is 1, the Ehrhart polynomial of $P$ is not magic positive. Therefore, the Ehrhart polynomial of $(0,1)$-polytope is never magic positive except for $[0,1]^d$. Furthermore, using the following proposition, we see that if a lattice polytope $P$ is reflexive and has volume less than 2, then its Ehrhart polynomial is never magic positive.
\begin{prop}[{\cite[Proposition 4.11]{beck2019h}}]
Let $P$ be a $d$-dimensional lattice polytope and let 
$$E_{P}(n) = \sum _{i=0}^da_in^i(1+n)^{d-i}$$
be its Ehrhart polynomial. Then $P$ is reflexive if and only if $a_j=a_{d-j}$ for any $j$.
\end{prop}

In this section, we discuss $\mindex(P)$ of some examples  of $(0,1)$-polytopes and reflexive polytopes $P$. 

\subsection{Standard simplex}\label{standard}
    The $d$-dimensional \textit{standard simplex}, denoted by $\Delta_d$, is the convex hull of $\eb_1, \eb_2, \ldots, \eb_{d+1}$ in $\RR^{d+1}$:
    \[
    \Delta_d := \conv\{\eb_1, \eb_2, \ldots, \eb_{d+1}\}.
    \]
    The Ehrhart polynomial of $\Delta_d$ is given by:
    \[
    E_{\Delta_d}(x) = \binom{x+d}{d}.
    \]
    \begin{prop}\label{standard:m-index}
        We have $\mindex(\Delta_d) = d$.
    \end{prop}
    \begin{proof}
        The Ehrhart polynomial of $k\Delta_d$ is as follows:
        \[
        E_{k\Delta_d}(x) = E_{\Delta_d}(kx) = \binom{kx+d}{d} = \frac{1}{d!}\prod_{j=0}^{d-1}((d-j)(x+1) + (k-d+j)x).
        \]
        Since $0 \leq j \leq d-1$, we have $1 \leq d-j \leq d$. Therefore, we have $\mindex(\Delta_d) = d$.
    \end{proof}

\subsection{Base polytope of minimal matroid}\label{minimal}
    We refer the reader to \cite{oxley2011matroid} for the introduction to matroids. It is known that if $M$ is a connected matroid with $n$ elements of rank $k$, then $|\Bc(M)| \geq k(n-k)+1$, where $\Bc(M)$ is the set of bases of $M$. Furthermore, there is a unique (up to isomorphism) connected matroid of size $n$ of rank $k$ for which equality holds. See, e.g., \cite{bruter1971extremal,murty1971number}. These matroids are referred to as the \textit{minimal matroids} denoted by $T_{k,n}$. It is clear from the minimality property of these matroids that the dual of the minimal matroid $T_{k,n}$ is isomorphic to $T_{n-k,n}$. The \textit{base polytope} of a matroid $M$, denoted by $B(M)$, is the convex hull of all indicator vectors in $\ZZ^n$, that is $e_B := \sum_{b \in B}e_b$ for $B \in \Bc(M)$.
    
    The Ehrhart polynomial of $B(T_{k,n})$, denoted by $D_{k,n}(x)$, is calculated in \cite{ferroni2022ehrhart} as follows:    
    \[
    D_{k,n}(x) = \frac{1}{\binom{n-1}{k-1}}\binom{x+n-k}{n-k}\sum_{j=0}^{k-1}\binom{n-k-1+j}{j}\binom{x+j}{j}.
    \]
    In \cite{ferroni2022ehrhart}, it is known that $B(T_{k,n})$ is Ehrhart positive.

    For the proof of Proposition~\ref{mindex:minimal}, we use the following lemma.
    \begin{lem}\label{magicproduct}
        Let $f(x)$ and $g(x)$ be polynomials with positive real coefficients of degree $n$ and $m$, respectively. If $f(x)$ and $g(x)$  are magic positive, then so is $f(x)g(x)$.
    \end{lem}
    \begin{proof}
        Since $f(x)$ and $g(x)$ are magic positive, there exist $a_0, \ldots, a_n, b_0, \ldots, b_m \geq 0$ such that
        \[
        f(x) = \sum_{i=0}^{n}a_ix^i(x+1)^{n-i} \hspace{5pt}\text{and}\hspace{5pt} g(x) = \sum_{i=0}^{m}b_ix^i(x+1)^{m-i}.
        \]
        Then we have
        \[
        f(x)g(x) = \sum_{k=0}^{n+m}\left(\sum_{i+j=k}a_ib_j\right)x^k(x+1)^{n+m-k},
        \]
        where we set $a_i = 0$ for $i > n$ and $b_j = 0$ for $ j > m$. Since $a_i, b_j \geq 0$ for all $i, j$, it follows that $f(x)g(x)$ is magic positive. 
\end{proof}
    \begin{prop}\label{mindex:minimal}
        We have the following.
        \begin{itemize}
            \item [(1)] For $n \geq 2$, we have $\mindex(B(T_{1,n})) = n-1$.
            \item [(2)] For $n \geq 4$, we have $\mindex(B(T_{2,n})) = n-2$.
        \end{itemize}
    \end{prop}
    \begin{proof}
        \begin{itemize}
            \item [(1)] Since $B(T_{1,n}) = \Delta_{n-1}$, we have $\mindex(B(T_{1,n})) = n-1$ by Proposition~\ref{standard:m-index}.
            \item [(2)] First, we verify that the Ehrhart polynomial of $(n-3)B(T_{2,n})$ is not magic positive. The Ehrhart polynomial is given by:
            \begin{align*}
                D_{2,n}((n-3)x) &= \frac{1}{n-1}\binom{(n-3)x+n-2}{n-2}\sum_{j=0}^{1}\binom{n-3+j}{j}\binom{(n-3)x+j}{j} \\
                & = \frac{1}{n-1}\binom{(n-3)x+n-2}{n-2}\left((n-1)(x+1)+(n^2-6n+7)x\right).
            \end{align*}
            When $n = 4$, we have
            \[
            D_{2,4}(x) = (x+1)^3 - \frac{5}{6}x(x+1)^2 + \frac{1}{6}x^2(x+1) + x^3,
            \]
            which is not magic positive. For $n \geq 5$, note that
            \[
            \binom{(n-3)x+n-2}{n-2} = \prod_{i=0}^{n-3}((n-2-i)(x+1) + (i-1)x).
            \]
            In the expansion of this product with respect to the basis $\{x^i(x+1)^{n-2-i}\}_{i=0}^{n-2}$, the coefficient of $x^{n-2}$ is 0, while the coefficient of $x^{n-3}(x+1)$ is $-(n-4)!$. Therefore, in the expansion of $D_{2,n}((n-3)x)$ with respect to the basis $\{x^i(x+1)^{n-1-i}\}_{i=0}^{n-1}$, the coefficient of $x^{n-2}(x+1)$ is $-(n^2-6n+7)(n-4)! < 0$. Hence, $D_{2,n}((n-3)x)$ is not magic positive.
            
            Next, we prove that the Ehrhart polynomial of $(n-2)B(T_{2,n})$ is magic positive. The Ehrhart polynomial is given by:
            \begin{align*}
                D_{2,n}((n-2)x) &= \frac{1}{n-1}\binom{(n-2)x+n-2}{n-2}\sum_{j=0}^{1}\binom{n-3+j}{j}\binom{(n-2)x+j}{j} \\
                & = \frac{1}{n-1}\binom{(n-2)x+n-2}{n-2}\left((n-1)(x+1) + (n^2-5n+5)x)\right).
            \end{align*}
            From Proposition~\ref{standard:m-index}, $\binom{(n-2)x+n-2}{n-2}$ is magic positive and $((n-1)(x+1) + (n^2-5n+5)x)$ is magic positive for all $n \geq 4$. Therefore, by Lemma~\ref{magicproduct}, we have $D_{2,n}((n-2)x)$ is magic positive.
        \end{itemize}
    \end{proof}
    
    Through computational experiments, we conjecture the following.
    \begin{q}
        For $1 \leq k \leq \bigfrac{n}{2}$, do we have $\mindex(B(T_{k,n})) = n- k$? 
    \end{q}

\subsection{Edge polytope of complete multipartite graph}
    Given a finite graph $G$ on $[d]$ with edge set $E(G)$, the \textit{edge polytope} is defined to be the convex hull of $\{\eb_i + \eb_j : \{i,j\} \in E(G)\}$.
    Let $q_1, q_2, \ldots, q_n$ denote a sequence of positive integers with $q_1 + q_2 + \cdots + q_n = d$ and
    let $V_1, V_2, \ldots, V_n$ denote a partition of $[d]$, that is, each $\emptyset \neq V_i \subset [d], V_i \cap V_j = \emptyset$ if $i \neq j$, and $[d] = V_1 \cup V_2 \cup \cdots \cup V_n$ with $|V_i| = q_i$, where $|V_i|$ is the cardinality of $V_i$ as a finite set.
    The \textit{complete multipartite graph of type $\mathbf{q} = (q_1, q_2, \ldots, q_n)$} is the finite graph $G_{\mathbf{q}}$ on the vertex set $[d] = V_1 \cup V_2 \cup \dots \cup V_n$ with the edge set
    \[
    E(G_{\mathbf{q}}) = \{\{k, \ell\} \mid k \in V_i, \ell \in V_j, 1 \leq i < j \leq n\}.
    \]
    Note that the edge polytope of a complete multipartite graph is the matroid polytope of a loopless matroid of rank 2. In particular, the edge polytope of a complete graph is the second hypersimplex.

    Let $E(G_{\qb}, x) = E_{P_{G_{\qb}}}(x)$.
    In \cite{ohsugi2000compressed}, the Ehrhart polynomial $E(G_{\qb}, x)$ of the edge polytope $P_{G_{\qb}}$ of the complete multipartite graph $G_{\qb}$ of type $\qb = (q_1, q_2, \ldots, q_n)$ on the vertex set $[d]$ with $n \geq 2$ is given as follows:
    \[
    \binom{2x+d-1}{d-1} - \sum_{k=1}^{n}\sum_{1 \leq i \leq j \leq q_k}\binom{x+j-i-1}{j-i}\binom{x+d-j-1}{d-j}.
    \]
    In \cite{ferroni2022ranktwo}, it is known that $P_{G_{\qb}}$ is Ehrhart positive.   
    In particular, if $G_{\qb}$ is a complete bipartite graph, then we have the following:
    \[
    E(G_{q_1,q_2},x) = \binom{x+q_1-1}{q_1-1}\binom{x+q_2-1}{q_2-1}.
    \]
    By Proposition~\ref{standard:m-index} and Lemma~\ref{magicproduct}, the following follows.
    \begin{prop}
        We have $\mindex(P_{G_{q_1,q_2}}) = \max\{q_1-1, q_2-1\}$.
    \end{prop}
    The following table presents the values of $\mindex(P_{G_{\mathbf{q}}})$.
    \begin{table}[H]
        \centering
        \begin{tabular}{|c||c|c|c|c|c|c|c|c|}
        \hline
          $\mathbf{q}$ & (1,1,1) & (2,2,2) & (3,3,3) & (1,2,3) & (1,2,4) & (1,2,5) & (1,2,3,4) & (1,1,1,5) \\ \hline
        $\mindex(P_{G_{\mathbf{q}}})$ & 2 & 2 & 4 & 3 & 4 & 5 & 5 & 5 \\ \hline
    \end{tabular}
    \end{table}
    Through computational experiments, we conjecture the following.
    \begin{q}
        For $d \geq 3$, is $\mindex(P_{G_{\mathbf{q}}})$ equal to one of $\max\{q_1, \ldots, q_d\}, \bigfrac{q_1 + \cdots +q_d}{2}$ or $\left\lceil\frac{q_1 + \cdots +q_d}{2}\right\rceil$?
    \end{q}

\subsection{Hypersimplex}\label{hypersimplex}
    Let us fix two positive integers $n$ and $k$ with $k \leq n$. The $(k,n)$-\textit{hypersimplex}, denoted by $\Delta_{k,n}$ is defined by:
    \[
    \Delta_{k,n} := \left\{(x_1,\ldots,x_n) \in [0,1]^n : \sum_{i=1}^{n}x_i = k\right\}.
    \]
    The hypersimplex $\Delta_{k,n}$ is also the matroid base polytope for a uniform matroid with $n$ elements of rank $k$.
    Moreover, $\Delta_{2,n}$ coincides with the edge polytope of the complete graph on $n$ vertices.
    In \cite{katzman2005hilbert}, the Ehrhart polynomial of $\Delta_{k,n}$, denoted by $E_{k,n}(x)$, is computed and is given by the following:
    \[
    E_{k,n}(x) = \sum_{j=0}^{k-1}(-1)^j\binom{n}{j}\binom{(k-j)x+n-1-j}{n-1}.
    \]
    In \cite{ferroni2021hypersimplices}, it is known that $\Delta_{k,n}$ is Ehrhart positive.

    \begin{prop}
        We have the following.
        \begin{itemize}
            \item [(1)] For $n \geq 2$, we have $\mindex(\Delta_{1,n}) = n-1$.
            \item [(2)] For $n \geq 4$, the Ehrhart polynomials of the $\bigfrac{n}{2}\Delta_{2,n}$ is magic positive. In particular, $\mindex(\Delta_{2,n}) \leq \bigfrac{n}{2}$.
            \item [(3)] For $n \geq 6$, the Ehrhart polynomials of the $\bigfrac{n}{2}\Delta_{3,n}$ is magic positive. In particular, $\mindex(\Delta_{3,n}) \leq \bigfrac{n}{2}$.
        \end{itemize}
    \end{prop}
    \begin{proof}
        \begin{itemize}
            \item [(1)] Since $\Delta_{1,n} = \Delta_{n-1}$, we have $\mindex(\Delta_{1,n}) = n-1$ by Proposition~\ref{standard:m-index}.
            \item [(2)] The Ehrhart polynomial of $k\Delta_{2,n}$ is given by:
            \[
            E_{2,n}(kx) = \binom{2kx+n-1}{n-1} - \binom{n}{1}\binom{kx+n-2}{n-1}.
            \]
            Let $E_{2,n}(kx) = \sum_{m=0}^{n-1}a_m(x+1)^{n-1-m}x^m$. For $1 \leq i \leq 2$, define the $2 \times (n-1)$ matrix
            \[
            A_{i,2} =
            \begin{pmatrix}
                2-i & 3-i & \cdots & n-1-i & n-i\\
                (3-i)k+i-2 & (3-i)k+i-3 & \cdots & (3-i)k-n+i+1 & (3-i)k-n+i
            \end{pmatrix}
            \]
            and for $I \subset [n-1] = \{1,2,\ldots,n-1\}$, define
            \[
            B_i^I = (-1)^{i-1}\binom{n}{i-1}\prod_{j \in I}b_{i2j}\prod_{j \in [n-1] \backslash I}b_{i1j}, \hspace{0.2cm} C_{I} = \sum_{i=1}^{2}B_i^I
            \]
            where $b_{ijk}$ denotes the $(j,k)$-entry of the matrix $A_{i,2}$. Then, $a_\ell = \frac{1}{(n-1)!}\sum_{I \in \binom{[n-1]}{\ell}}C_I$ for all $0 \leq \ell \leq n-1$. Hence, to prove that the Ehrhart polynomial of $k\Delta_{2,n}$ is magic positive, it suffices to show that $C_I \geq 0$ for all $I \in \binom{[n-1]}{\ell}$,
            Note that $B_1^I > 0$ and if $1 \notin I$, then $C_I = B_1^I > 0$. We now assume $k = \bigfrac{n}{2}$. If $I = \{1\}$, then
            \[
            C_{\{1\}} = \left(2\bigfrac{n}{2}-1\right)(n-1)! - n\bigfrac{n}{2}(n-2)! = (n-2)!\left(\bigfrac{n}{2}(n-2)-(n-1)\right) > 0.
            \]
            For $i \geq 2$, observe that
            \[
            \frac{2\bigfrac{n}{2}-i}{i}-\frac{\bigfrac{n}{2}-i+1}{i-1} = \frac{\bigfrac{n}{2}(i-2)}{i(i-1)} \geq 0,
            \]
            and also $(2\bigfrac{n}{2}-i)/i > 0$ for all $i \in [n-1]$. Thus,
            \[
            C_{\{1,i\}} = \left(2\bigfrac{n}{2}-1\right)(n-1)!\frac{2\bigfrac{n}{2}-i}{i} - n\bigfrac{n}{2}(n-2)!\frac{\bigfrac{n}{2}-i+1}{i-1} > 0.
            \]
            By a similar argument, we can show that $C_I > 0$ for any $I \in \binom{[n-1]}{\ell}$. Therefore, the Ehrhart polynomial of $\bigfrac{n}{2}\Delta_{2,n}$ is magic positive.
            \item [(3)] The Ehrhart polynomial of $k\Delta_{3,n}$ is given by:
            \[
            E_{3,n}(kx) = \binom{3kx+n-1}{n-1} - \binom{n}{1}\binom{2kx+n-2}{n-1} + \binom{n}{2}\binom{kx+n-3}{n-1}.
            \]
            Let $E_{3,n}(kx) = \sum_{m=0}^{n-1}a_m(x+1)^{n-1-m}x^m$. For $1 \leq i \leq 3$, define the $2 \times (n-1)$ matrix
            \[
            A_{i,3} =
            \begin{pmatrix}
                2-i & 3-i & \cdots & n-1-i & n-i\\
                (4-i)k+i-2 & (4-i)k+i-3 & \cdots & (4-i)k-n+i+1 & (4-i)k-n+i
            \end{pmatrix}
            \]
            and for $I \subset [n-1]$, define
            \[
            B_i^I = (-1)^{i-1}\binom{n}{i-1}\prod_{j \in I}b_{i2j}\prod_{j \in [n-1] \backslash I}b_{i1j}, \hspace{0.2cm} C_{I} = \sum_{i=1}^{3}B_i^I
            \]
            where $b_{ijk}$ denotes the $(j,k)$-entry of the matrix $A_{i,3}$. Similarly to (2), we have $a_\ell = \frac{1}{(n-1)!}\sum_{I \in \binom{[n-1]}{\ell}}C_I$ for all $0 \leq \ell \leq n-1$, and to prove that the Ehrhart polynomial of the $k\Delta_{3,n}$ is magic positive, it suffices to show that $C_I \geq 0$ for all $I \in \binom{[n-1]}{\ell}$. Note that $B_1^I > 0$ and if $1,2 \notin I$, then $C_I = B_1^I > 0$.
            We assume that $k = \bigfrac{n}{2}$. When $I = \{1\}, \{2\}, \{1,2\}$, $C_I$ is as follows for each case:
            \begin{align*}
                C_{\{1\}} &= (n-1)!\left(3\bigfrac{n}{2}-1\right)-\binom{n}{1}(n-2)!\left(2\bigfrac{n}{2}\right) = (n-2)!\left(\bigfrac{n}{2}(n-3) - (n-1)\right) > 0, \\
                C_{\{2\}} &= (n-1)!\frac{3\bigfrac{n}{2}-2}{2}+\binom{n}{2}(-1)(n-3)!\bigfrac{n}{2} = (n-1)(n-3)!\left(\bigfrac{n}{2}(n-3)-(n-2)\right) > 0,\\
                C_{\{1,2\}} &= (n-1)!\frac{3\bigfrac{n}{2}-2}{2}\left(3\bigfrac{n}{2}-1\right)-\binom{n}{1}(n-2)!\left(2\bigfrac{n}{2}\right)\left(2\bigfrac{n}{2}-1\right)\\
                &\quad+\binom{n}{2}(n-3)!\bigfrac{n}{2}\left(\bigfrac{n}{2}+1\right) \\
                &= (n-3)!\left(\bigfrac{n}{2}(n-3) - (n-1)\right)\left(\bigfrac{n}{2}(n-3) - (n-2)\right) > 0.
            \end{align*}
            
            For $i \geq 3$, since we have
            \[
            \frac{3\bigfrac{n}{2}-i}{i} \geq \frac{2\bigfrac{n}{2}-i+1}{i-1} \geq \frac{\bigfrac{n}{2}-i+2}{i-2},
            \]
            and $(3\bigfrac{n}{2}-i)/i > 0$ by a similar argument to $\bigfrac{n}{2}\Delta_{2,n}$, if $\{1,2\} \not\subset I$, then $C_I > 0$.
            Let $A = (n-1)!\frac{3\bigfrac{n}{2}-2}{2}(3\bigfrac{n}{2}-1), B = \binom{n}{1}(n-2)!(2\bigfrac{n}{2})(2\bigfrac{n}{2}-1), C = \binom{n}{2}(n-3)!\bigfrac{n}{2}(\bigfrac{n}{2}+1)$. For $i \geq 3$, if $I = \{1,2,i\}$, then we have
            \begin{align*}
                C_I = C_{\{1,2,i\}} &= \frac{3\bigfrac{n}{2}-i}{i}A-\frac{2\bigfrac{n}{2}-i+1}{i-1}B+\frac{\bigfrac{n}{2}-i+2}{i-2}C\\
                &=\frac{\bigfrac{n}{2}(i-3)}{i}\left(\frac{1}{i}B - \frac{2}{i-1}C\right) + \frac{3\bigfrac{n}{2}-i}{i}(A-B+C)
            \end{align*}
            and 
            \[
            \frac{C}{B} = \frac{\binom{n}{2}(n-3)!\bigfrac{n}{2}(2\bigfrac{n}{2}-1)}{n(n-2)!(2\bigfrac{n}{2})(2\bigfrac{n}{2}-1)} = \frac{(n-1)(\bigfrac{n}{2}+1)}{4(n-2)(2\bigfrac{n}{2}-1)} \leq \frac{(6-1)(3+1)}{4\times4(2\times3-1)}  = \frac{1}{4} \leq \frac{i-2}{2(i-1)}.
            \]
            Therefore, we have $C_{\{1,2,i\}} > 0$. By a similar argument, we can show that $C_I > 0$ for any $I \in \binom{[n-1]}{\ell}$. From the above discussion, we conclude that the Ehrhart polynomial of $\bigfrac{n}{2}\Delta_{3,n}$ is magic positive.
        \end{itemize}
    \end{proof}
    Through computational experiments, we conjecture the following.
    \begin{q}
        For $0 \leq k \leq \bigfrac{n}{2}-2$, do we have $\mindex(\Delta_{\bigfrac{n}{2}- k, n}) = k+2$?
    \end{q}

    It has been observed that it is sufficient to make their Ehrhart polynomials magic positive by dilating the polytopes introduced so far by a factor no less than their dimension. The following two examples illustrate cases where a polytope fails to exhibit magic positivity in its Ehrhart polynomial unless it is sufficiently dilated relative to its dimension.

\subsection{Cross polytope}\label{cross}
    The $d$-dimensional \textit{cross polytope}, denoted by $\lozenge_d$, is a polytope defined by 
    \[
    \lozenge_d := \conv\{\pm\eb_1, \pm\eb_2, \ldots, \pm\eb_d\}.
    \]
    The Ehrhart polynomial of $\lozenge_d$ is given by:
    \[
    E_{\lozenge_d}(x) = \sum_{i=0}^{d}2^i\binom{d}{i}\binom{x}{i},
    \]
    see, e.g., \cite{liu2019positivity}. The Ehrhart positivity of $\lozenge_d$ follows from the CL-ness. See Section~\ref{CL} for the CL-ness of the polytope.
    
    The following table presents the values of $\mindex(\lozenge_d)$. The table shows that to make the Ehrhart polynomial of a cross polytope magic positive, the polytope usually needs to be dilated beyond its dimension.
    \begin{table}[H]
    \centering
    \begin{tabular}{|c||c|c|c|c|c|c|c|c|c|c|c|c|c|c|c|c|}
        \hline
          $d$ & 1 & 2 & 3 & 4 & 5 & 6 & 7 & 8 & 9 & 10 & 11 & 12 & 13 & 14 & 15 & 16\\ \hline
        $\mindex(\lozenge_d)$ & 1 & 1 & 2 & 4 & 6 & 10 & 13 & 18 & 23 & 29 & 35 & 42 & 50 & 59 & 68 & 78\\ \hline
    \end{tabular}
    \end{table}

\subsection{Standard reflexive simplex}\label{sr}
    The $d$-dimensional \textit{standard reflexive simplex}, denoted by $\Delta_d'$, is the convex hull of 
    \[
    \Delta_d' := \conv\{\eb_1, \eb_2, \ldots, \eb_d, -(\eb_1 + \cdots + \eb_d)\}.
    \]
    The Ehrhart polynomial of $\Delta_d'$ is given by:
    \[
    E_{\Delta_d'}(n) = \sum_{i=0}^{d}\binom{n+d-i}{d}.
    \]
    
    As with the case of the cross polytope, the Ehrhart positivity of $\Delta_d'$ follows from the CL-ness.
    
    The following table presents the values of $\mindex(\Delta_d')$.
    \begin{table}[H]
    \centering
    \begin{tabular}{|c||c|c|c|c|c|c|c|c|c|c|c|c|c|c|c|c|}
        \hline
          $d$ & 1 & 2 & 3 & 4 & 5 & 6 & 7 & 8 & 9 & 10 & 11 & 12 & 13 & 14 & 15 & 16\\ \hline
        $\mindex(\Delta_d')$ & 1 & 2 & 4 & 10 & 20 & 34 & 55 & 83 & 119 & 163 & 218 & 284 & 361 & 452 & 557 & 677\\ \hline
    \end{tabular}
    \end{table}

\subsection{CL-polytope}\label{CL}
    A reflexive polytope $P$ is said to be a \textit{CL-polytope} if all of the complex roots of the Ehrhart polynomial $E_P(n)$ lie on the line $\mathrm{Re}(z) = -\frac{1}{2}$ in the complex plane, where $\mathrm{Re}(z)$ denotes the real part of a complex number $z$. 
    In \cite[Lemma~4.15]{ferroni2024examples}, it is known that if a reflexive polytope $P$ is a CL-polytope, then $P$ is Ehrhart positive.
    \begin{lem}\label{quadratic}
        Let $f(x)$ be a quadratic polynomial with roots $-\frac{1}{2}\pm b\sqrt{-1}$. Then $f(kx)$ is magic positive if and only if $k \geq \frac{1}{2} + 2b^2$.
    \end{lem}
    \begin{proof}
        Since the roots of $f(x)$ are $-\frac{1}{2}\pm b\sqrt{-1}$, we have
        \begin{align*}
            f(kx) &= a\left(kx + \frac{1}{2} + bi\right)\left(kx + \frac{1}{2} - bi\right) \\
            &= a\left(\left(\frac{1}{4}+b^2\right)(x+1)^2 + \left(k-\frac{1}{2}-2b^2\right)x(x+1) + \left(\left(k-\frac{1}{2}\right)^2+b^2\right)x^2\right),
        \end{align*}
        where $a$ is a positive real number. Therefore, $f(kx)$ is magic positive if and only if $k \geq \frac{1}{2}+2b^2$.
    \end{proof}    

    \begin{prop}[\cite{braun2008norm}]\label{bound}
        The roots $\alpha$ of an Ehrhart polynomial of degree $d$ satisfy the inequality $|\alpha + \frac{1}{2}| \leq d(d-\frac{1}{2})$.
    \end{prop}
    
    By Lemma~\ref{magicproduct}, Lemma~\ref{quadratic} and Proposition~\ref{bound}, the following follows.
    
    \begin{cor}
        Let $P$ be a CL-polytope of dimension $d$. Then $\mindex(P)$ admits an upper bound that depends only on $d$.
    \end{cor}
    \begin{proof}
        For real numbers $b_1, \ldots, b_{\bigfrac{d}{2}}$, define
        \[
        f_i(x) = \left(x + \frac{1}{2}+b_i\sqrt{-1}\right)\left(x + \frac{1}{2}-b_i\sqrt{-1}\right).
        \]
        Since $P$ is a CL-polytope, $E_P(x)$ can be expressed as
        \[
        E_P(x) = b(x)\prod_{i=1}^{\bigfrac{d}{2}} f_i(x),
        \]
        where for a positive rational number $a$, $b(x) = a$ if $d$ is even and $b(x) = a(2x+1)$ otherwise. 
        By Lemma~\ref{quadratic}, each $f_i(kx)$ is magic positive for all $k \geq \frac{1}{2} + 2b_i^2$.
        Let $b = \max\{b_0, \ldots, b_{\bigfrac{d}{2}}\}$. Then, by Lemma~\ref{magicproduct}, we have $\mindex(P) \leq \left\lceil \frac{1}{2} + 2b^2 \right\rceil$. Moreover, by Proposition~\ref{bound}, since $b \leq d(d - \frac{1}{2})$, it follows that $\mindex(P) \leq \left\lceil\frac{1}{2} + 2d^2(d-\frac{1}{2})^2\right\rceil$.
    \end{proof}
    In general, there is no general upper bound for the $\mindex$ of a polytope depending only on the dimension. However, for CL-polytopes, such a bound does exist.
    
    Through the examples in Chapter 5, we see that it seems difficult to explicitly determine $\mindex(P)$ for a lattice polytope $P$.

    \begin{rem}
        By Theorem~\ref{main1}, any polynomial with positive real coefficients becomes magic positive  when the polynomial is `dilated' in a certain sense. Therefore, by \cite[Theorem~1.1]{athanasiadis2025realeulerian}, the polynomial resulting from the Eulerian transformation is real-rooted.
        Moreover, see \cite[Section~6]{petter2022eulerian} for the relationship between Ehrhart theory and the Eulerian transformation.
    \end{rem}
\bibliographystyle{plain} 
\bibliography{ref}
\end{document}